\documentclass[reqno]{amsart}
\usepackage{amsmath}

\usepackage{amsfonts}
\usepackage{amssymb}
\usepackage{hyperref}

\usepackage[a4paper,top=3.5cm,bottom=3cm,left=3.5cm,right=3.5cm]{geometry}

\begin{document}
\title[Positive solutions of nonlocal fractional boundary value problem]{Positive solutions of nonlocal fractional boundary value problem involving Riemann-Stieltjes integral condition}
\author[F. Haddouchi]{Faouzi Haddouchi}
\address{
Department of Physics, University of Sciences and Technology of
Oran-MB, El Mnaouar, BP 1505, 31000 Oran, Algeria
\newline
And
\newline
Laboratoire de Math\'ematiques Fondamentales et Appliqu\'ees d'Oran (LMFAO). Universit\'e Oran1. B.P. 1524 El Mnaouer, Oran, Alg\'erie.}

\email{fhaddouchi@gmail.com}

\begin{abstract}
In this paper, we investigate the existence of positive solutions for a nonlocal fractional boundary value problem involving Caputo fractional derivative and nonlocal Riemann-Stieltjes integral boundary condition. By using the spectral analysis of the relevant linear operator and Gelfand's formula, we obtain an useful upper and lower bounds for the spectral radius. Our discussion is based on the properties of the Green's function and the fixed point index theory in cones.
\end{abstract}

\maketitle \numberwithin{equation}{section}
\newtheorem{theorem}{Theorem}[section]
\newtheorem{lemma}[theorem]{Lemma}
\newtheorem{definition}[theorem]{Definition}
\newtheorem{proposition}[theorem]{Proposition}
\newtheorem{corollary}[theorem]{Corollary}
\newtheorem{remark}[theorem]{Remark}
\newtheorem{exmp}{Example}[section]

\section{Introduction\label{sec:1}}
Differential equations with fractional derivative related to nonlocal conditions have been discussed as valuable tools in the modeling of many physical phenomena in natural sciences and engineering, such as physics, chemistry, biology, economics, aerodynamics, viscoelasticity, control theory, earthquake, traffic flow, and various material processes \cite{Tarasov,Yin,Paola,Hartley,Caputo,Saka,Caffarelli,Ivancevic,Das1,Mongiovi,Zhang4,Zhang5,Das2,Sabatier,Kilbas2006,
Beleanu,Podlubny,Zhang3}.

Recently, some interesting results about the existence of positive solutions for nonlinear fractional differential equations involving the Riemann-Stieltjes integral boundary condition have been reported \cite{Hao2,Guo,Debowska,Tan,Ren,Jia,Srivastava,Ahmed,Padhi,Jankowski1,Jankowski2,Wang,Wang2,Kaihong,Kossowski,Jiang}.

 In \cite{Hao}, the authors established the existence of positive solutions for the $n$th-order singular nonlocal boundary value problem
       \begin{equation}\label{eq1}
       \begin{cases}u^{(n)}(t) + \lambda a(t)f(t,u(t))=0, \ 0< t<1, \\
      u(0)= u^{\prime}(0)=\ldots =u^{(n-2)}(0)=0,\ u(1)=\int_{0}^{1}u(s)dA(s).
       \end{cases}
       \end{equation}
where $a(t)$ have singularity, $f(t,x)$ may also have singularity at $x=0$, and $\int_{0}^{1}u(s)dA(s)$ is given by a Riemann-Stieltjes integral with a signed measure.

In \cite{Wang}, Y. Wang et al., studied the following singular boundary value problem of a nonlocal fractional differential
equation

  \begin{equation}\label{eq2}
    \begin{cases} D^{\alpha}u(t) +q(t)f(t,u(t)) = 0, \ 0< t<1,\ n-1<\alpha\leq n, \\
   u(0)= u^{\prime}(0)=\ldots =u^{(n-2)}(0)=0,\ u(1)=\int_{0}^{1}u(s)dA(s),
    \end{cases}
    \end{equation}
   where $\alpha\geq2$, $D^{\alpha}$ is the standard Riemann-Liouville derivative. The existence and multiplicity  of positive
solutions are obtained by means of the fixed point index theory in cones.

 In \cite{Song}, Q. Song and Z. Bai considered the following boundary value problem of fractional differential equation
 with Riemann-Stieltjes integral boundary condition

  \begin{equation}\label{eq3}
    \begin{cases} D^{\alpha}u(t) +\lambda f(t,u(t),u(t)) = 0, \ 0< t<1,\ n-1<\alpha\leq n, \\
   u^{(k)}(0)=0,\ 0\leq k\leq n-2,\ u(1)=\int_{0}^{1}u(s)dA(s),
    \end{cases}
    \end{equation}
   where $D^{\alpha}$ is the Riemann-Liouville derivative. They used a fixed point theorem   and the properties of mixed monotone operator theory to get the existence and uniqueness of positive solutions.

 In \cite{Lv}, T. Lv et al., investigated the following fractional differential equation with multistrip
 Riemann-Stieltjes integral boundary conditions

\begin{equation}\label{eq4}
  \begin{cases} D^{\alpha}u(t) + f(t,u(t),D^{\beta}u(t)) = 0, \ t\in(0,1), \\   u(0)=D^{\beta}u(0)=0,\ u(1)=\sum_{i=1}^{m}\int_{I_{i}}u(s)dA(s),
  \end{cases}
  \end{equation}
   where $2<\alpha\leq 3$, $0<\beta< 1$, $I_{i}\subset (0,1)$, $i\in\{1,\dots ,m\}$, and $D^{\alpha}$
 is the standard Riemann-Liouville derivative. By using the Leggett-Williams fixed point theorem, the existence result of positive solutions is obtained.

In \cite{Zhang1}, X. Zhang et al., studies the existence and uniqueness
 of positive solutions of the Riemann-Liouville fractional differential equation

\begin{equation}\label{eq5}
  \begin{cases} D^{\alpha}u(t) + f(t,u(t),D^{\alpha-1}u(t),D^{\alpha-2}u(t),\dots,D^{\alpha-n+1}u(t)) = 0,
 \ t\in(0,1), \\
D^{\alpha-i}u(0)=0,\ i=3,4,\dots,n,\\
 D^{\alpha-2}u(0)=\int_{0}^{1}u(s)dB_{1}(s),\\
D^{\alpha-1}u(1)=\int_{0}^{1}u(s)dB_{2}(s),
  \end{cases}
  \end{equation}
where $D^{\alpha}$, $D^{\alpha-i}$ are the Riemann-Liouville fractional derivatives, $n-1<\alpha\leq n$, $n\geq3$ ($n\in\mathbb{N}$),
$f:[0,1]\times\mathbb{R}^{n}_{+}\rightarrow \mathbb{R}_{+}$ is continuous, $\mathbb{R}_{+}=[0,\infty)$,
$\int_{0}^{1}u(s)dB_{i}(s)$ are Riemann-Stieltjes integrals, where $B_{i}$ ($i=1,2$) are nondecreasing functions. The existence results
are abtained via the use of fixed point theorems on cones in partially ordered Banach spaces.

 In \cite{Zhang2}, X. Zhang et al., used the fixed point index theory in cones to study the existence of positive solutions to the singular fractional
differential equation with signed measure

\begin{equation}\label{eq6}
  \begin{cases} (-D^{\alpha}u)(t)=f(t,u(t),D^{\beta}u(t)), \ t\in(0,1), \\
  D^{\beta}u(0)=0,\  D^{\beta}u(1)=\int_{0}^{1}D^{\beta}u(s)dA(s),
  \end{cases}
  \end{equation}
where $D^{\alpha}$, $D^{\beta}$ are the Riemann-Liouville fractional derivatives, $\int_{0}^{1}u(s)dA(s)$ is
denoted by a Riemann-Stieltjes integral and $0<\beta\leq1$, $1<\alpha\leq 2$, $\alpha-\beta>1$, $A$ is a function of bounded variation
and $dA$ can be a signed measure, $f(t,x,y)$ may be singular at both $t=0,1$ and $x=y=0$.

In \cite{Cabada}, A. Cabada and G. Wang obtained a sufficient condition for the existence
of a positive solution to the fractional boundary value problem

\begin{equation}\label{eq7}
  \begin{cases} D^{\alpha}u(t)+f(t,u(t))=0, \ t\in(0,1),\ 2<\alpha<3 \\
 u(0)=u^{\prime\prime}(0)=0,\  u(1)=\lambda \int_{0}^{1}u(s)ds,\ 0<\lambda<2,
  \end{cases}
  \end{equation}
  where $D^{\alpha}$ is the Caputo fractional derivative and $f:[0,1]\times[0,\infty)\rightarrow [0,\infty)$ is a continuous function.

Motivated by the results mentioned above, and inspired by the work in \cite{Cabada}, in this paper, we are concerned with the following nonlocal fractional boundary value problem with a Riemann-Stieltjes integral boundary condition

\begin{equation}\label{eq8}
  \begin{cases}  D^{\alpha}u(t) +f(t,u(t)) = 0,\  t \in (0,1),\\
 u(0)= u^{\prime\prime} (0) =0,\ u(1)=\mu u(\eta)+ \beta{\gamma[u]},
  \end{cases}
  \end{equation}
where the function $f:[0,1]\times [0,\infty)\rightarrow [0,\infty)$ is continuous, $2<\alpha\leq3$, $\eta\in(0,1)$, $\mu,\beta\geq0$, $D^{\alpha}$ is the standard Caputo derivative. $\gamma$ is a continuous linear functional given by the Riemann-Stieltjes integral
 \[\gamma[u]=\int_{0}^{1}u(s)dA(s),\]
 with the function $A:[0,1]\rightarrow\mathbb{R}$ of bounded variation. $A$ can includes both sums and integrals, which implies that the nonlocal boundary conditions of Riemann-Stieltjes integral type is a more general case than the multi-point boundary condition and the integral boundary condition.  In the special case when $A(s)=s$, $\beta\in(0,2)$, and $\mu=0$, our problem reduces to the boundary value problem \eqref{eq7}. It is worth noting that our results generalize and extend results in \cite{Cabada}.

 The rest of the paper is organized as follows. In section 2, we present some theorems and lemmas that will be used to prove our main results. In section 3, we investigate the existence of at least one positive solution for \eqref{eq8}, and an example is given to illustrate our result.

 \section{Preliminaries}
 \begin{definition}
 Let $E$ be a real Banach space. A nonempty, closed, convex set $
 K\subset E$ is a cone if it satisfies the following two conditions:
 \begin{itemize}
\item[(i)]$\lambda K\subset K \ \text{for all}\ \lambda\geq0$;
\item[(ii)] $K\cap (-K)=\{0\}$.
 \end{itemize}
 If $\overline{K-K}=E$, i.e., the set $\{u-v: u,v\in K\}$ is dense in $E$, then $K$ is called a total cone. If $K-K=E$, $K$ is called a reproducing cone.
 \end{definition}

 \begin{definition}
 An operator $T:E\rightarrow E$ is completely continuous if it is continuous
 and maps bounded sets into relatively compact sets.
 \end{definition}
 \begin{definition}
The Riemann-Liouville fractional integral of order $\alpha$ for a continuous function $f$ is defined as
\begin{equation*}
I^{\alpha}f(t)=\frac{1}{\Gamma(\alpha)}\int_{0}^{t}\frac{f(s)}{(t-s)^{1-\alpha}}ds, \ \alpha>0,
\end{equation*}
provided the integral exists, where $\Gamma(.)$ is the gamma function, which is defined by $\Gamma(x)=\int_{0}^{\infty}t^{x-1}e^{-t}dt$.
\end{definition}

\begin{definition}
For at least n-times continuously differentiable function $f:[0,\infty)\rightarrow \mathbb{R}$, the Caputo derivative of fractional order $\alpha$ is defined as
\begin{equation*}
^{c}D^{\alpha}f(t)=\frac{1}{\Gamma(n-\alpha)}\int_{0}^{t}\frac{f^{(n)}(s)}{(t-s)^{\alpha+1-n}}ds,\ n-1<\alpha<n,\ n=[\alpha]+1,
\end{equation*}
where $[\alpha]$ denotes the integer part of the real number $\alpha$.
\end{definition}

\begin{lemma}[\cite{Kilbas2006}] \label{lem 2.1}
For $\alpha> 0$, the general solution of the fractional differential equation $^{c}D^{\alpha}x(t)=0$ is
given by
\begin{equation*}
x(t)=c_{0}+c_{1}t+...+c_{n-1}t^{n-1},
\end{equation*}
where $c_{i}\in \mathbb{R}$, $i = 0, 1,...,n-1 \ (n = [\alpha] + 1)$.
\end{lemma}
According to Lemma \ref{lem 2.1}, it follows that
\begin{equation*}
{I^{\alpha}}\ {^{c}D^{\alpha}}x(t)=x(t)+c_{0}+c_{1}t+...+c_{n-1}t^{n-1},
\end{equation*}
for some $c_{i}\in \mathbb{R}$, $i = 0, 1,...,n-1 \ (n = [\alpha] + 1)$.

To prove our results, the following lemmas are needed in our discussion.

\begin{lemma}[\cite{Guo}] \label{lem 2.2}
Let $\mathcal{P}$ be a cone in Banach space $E$, and $\Omega$ be a bounded open subset of $E$ with
$\theta\in\Omega$. Suppose that $A:\mathcal{P}\cap\overline{\Omega}\rightarrow\mathcal{P}$ is a completely continuous
operator. If $\mu Au\neq u$ for every $u\in\mathcal{P}\cap\partial\Omega$ and $0 <\mu\leq 1$, then $i(A,\mathcal{P}\cap\Omega,\mathcal{P}) = 1$.
\end{lemma}
\begin{lemma}  [\cite{Guo}] \label{lem 2.3}
Let $\mathcal{P}$ be a cone in Banach space $E$, and $\Omega$ be a bounded open subset of $E$. Suppose that $A:\mathcal{P}\cap\overline{\Omega}\rightarrow\mathcal{P}$ is a completely continuous
operator. If there exists $u_{0}\in \mathcal{P}\setminus\{0\}$ such that $u-Au\neq \mu u_{0}$ for every $u\in\mathcal{P}\cap\partial\Omega$ and $\mu\geq0$, then $i(A,\mathcal{P}\cap\Omega,\mathcal{P}) =0$.
\end{lemma}

\begin{lemma}\label{lem 2.4} (Krein-Rutman, see\emph{\cite[Theorem 19.2]{ Deimling}})
Let $\mathcal{P}$ be a total cone in a real Banach space $E$ and let $L:E\rightarrow E$ be a compact linear operator with $L(\mathcal{P})\subset \mathcal{P}$. If $r(L)>0$, then there exists $\varphi\in \mathcal{P}\setminus\{0\}$ such that $L\varphi=r(L)\varphi$, where $r(L)$ is the spectral radius of $L$.
\end{lemma}

 Consider the nonlocal fractional boundary value problem
 \begin{equation}\label{eq9}
 D^{\alpha}u(t) +h(t) = 0,\  t \in (0,1),
 \end{equation}
 \begin{equation}\label{eq10}
  u(0)= u^{\prime\prime} (0) =0,\ u(1)=\mu u(\eta)+ \beta{\gamma[u]},
 \end{equation}
where $\alpha\in(2,3]$ and $h\in C[0,1]$.

For convenience, we denote $\Lambda= \mu \eta+ \beta \gamma[t]$. 
\begin{lemma}\label{lem 2.5}
Let $\Lambda \neq 1 $. Then for any $ h \in C[0,1] $, the fractional boundary value problem \eqref{eq9}-\eqref{eq10} has a unique solution which can be expressed by
$$ u(t) = \int_{0}^{1}  \mathcal{H}(t,s)h(s) ds ,$$
where
\begin{equation}\label{eq11}
 \mathcal{H}(t,s)=\frac{\beta {t}}{1-\Lambda}\int_{0}^{1}\mathcal{G}(t,s)dA(t)+\frac{\mu{t}}{1-\Lambda} \mathcal{G}(\eta,s)+\mathcal{G}(t,s),
 \end{equation}
and
\begin{equation}\label{eq12}
\mathcal{G}(t,s)=\frac{1}{\Gamma(\alpha)} \begin{cases}t(1-s)^{\alpha-1}-(t-s)^{\alpha-1}, & 0 \leq s \leq t \leq 1; \\
t(1-s)^{\alpha-1} , & 0 \leq t \leq s \leq 1.
\end{cases}
\end{equation}
\end{lemma}

\begin{proof}
By applying Lemma \ref{lem 2.1} the equation \eqref{eq9} is equivalent to the following integral equation
\begin{equation*}
u(t)=-\int_{0}^{t}\frac{(t-s)^{\alpha-1}}{\Gamma(\alpha)}h(s)ds+c_{1}+c_{2}t+c_{3}t^{2},
\end{equation*}
where $c_{1}, c_{2}, c_{3} \in \mathbb{R}$ are arbitrary constants.
Conditions $u(0)=0$ and $u^{\prime \prime}(0)=0$ imply that $c_{1}=c_{3}=0$, i.e.,
\begin{equation*}
u(t)=-\int_{0}^{t}\frac{(t-s)^{\alpha-1}}{\Gamma(\alpha)}h(s)ds+c_{2}t.
\end{equation*}
In particular, we obtain
\begin{equation*}
c_{2}=u(1)+\int_{0}^{1}\frac{(1-s)^{\alpha-1}}{\Gamma(\alpha)}h(s)ds.
\end{equation*}
Therefore, we have
\begin{equation}\label{eq13}
u(t)=-\int_{0}^{t}\frac{(t-s)^{\alpha-1}}{\Gamma(\alpha)}h(s)ds+tu(1)+
\int_{0}^{1}\frac{t(1-s)^{\alpha-1}}{\Gamma(\alpha)}h(s)ds.
\end{equation}
From \eqref{eq13} and \eqref{eq10}, we get
\begin{eqnarray*}
u(1)& =&\mu \bigg(-\int_{0}^{\eta}\frac{(\eta-s)^{\alpha-1}}{\Gamma(\alpha)}h(s)ds+\eta u(1)+
\int_{0}^{1}\frac{\eta(1-s)^{\alpha-1}}{\Gamma(\alpha)}h(s)ds\bigg)\\
&&+\beta\int_{0}^{1}\bigg(-\int_{0}^{t}\frac{(t-s)^{\alpha-1}}{\Gamma(\alpha)}h(s)ds+tu(1)
+\int_{0}^{1}\frac{t(1-s)^{\alpha-1}}{\Gamma(\alpha)}h(s)ds\bigg)dA(t),
\end{eqnarray*}
which, on solving, yields
\begin{eqnarray}\label{eq14}
u(1)& =&\frac{-\mu}{1-\Lambda}\int_{0}^{\eta}\frac{(\eta-s)^{\alpha-1}}{\Gamma(\alpha)}h(s)ds+
\frac{ \mu \eta}{1-\Lambda}\int_{0}^{1}\frac{(1-s)^{\alpha-1}}{\Gamma(\alpha)}h(s)ds \nonumber\\
&&-\frac{\beta}{1-\Lambda}\int_{0}^{1}\bigg(\int_{0}^{t}\frac{(t-s)^{\alpha-1}}{\Gamma(\alpha)}h(s)ds\bigg)dA(t)\\
&&+\frac{\beta}{1-\Lambda}\int_{0}^{1}\bigg(\int_{0}^{1}\frac{t(1-s)^{\alpha-1}}{\Gamma(\alpha)}h(s)ds\bigg)dA(t)\nonumber.
\end{eqnarray}
Replacing \eqref{eq14} in \eqref{eq13}, we get

\begin{eqnarray*}
u(t)& =&-\int_{0}^{t}\frac{(t-s)^{\alpha-1}}{\Gamma(\alpha)}h(s)ds+\int_{0}^{1}\frac{t(1-s)^{\alpha-1}}{\Gamma(\alpha)}h(s)ds
-\frac{\mu t}{1-\Lambda}\int_{0}^{\eta}\frac{(\eta-s)^{\alpha-1}}{\Gamma(\alpha)}h(s)ds\\
&&+\frac{\mu\eta}{1-\Lambda}\int_{0}^{1}\frac{t(1-s)^{\alpha-1}}{\Gamma(\alpha)}h(s)ds
-\frac{\beta t}{1-\Lambda}\int_{0}^{1}\bigg(\int_{0}^{t}\frac{(t-s)^{\alpha-1}}{\Gamma(\alpha)}h(s)ds\bigg)dA(t)\\
&&+\frac{\beta t}{1-\Lambda}\int_{0}^{1}\bigg(\int_{0}^{1}\frac{t(1-s)^{\alpha-1}}{\Gamma(\alpha)}h(s)ds\bigg)dA(t)\\
&=&\int_{0}^{t}\frac{t(1-s)^{\alpha-1}-(t-s)^{\alpha-1}}{\Gamma(\alpha)}h(s)ds
+\int_{t}^{1}\frac{t(1-s)^{\alpha-1}}{\Gamma(\alpha)}h(s)ds\\
&&+\frac{\mu t}{1-\Lambda}\int_{0}^{\eta}\frac{\eta(1-s)^{\alpha-1}-(\eta-s)^{\alpha-1}}{\Gamma(\alpha)}h(s)ds
+\frac{\mu t}{1-\Lambda}\int_{\eta}^{1}\frac{\eta(1-s)^{\alpha-1}}{\Gamma(\alpha)}h(s)ds\\
&&+\frac{\beta t}{1-\Lambda}\int_{0}^{1}\bigg(\int_{0}^{t}\frac{t(1-s)^{\alpha-1}-(t-s)^{\alpha-1}}{\Gamma(\alpha)}h(s)ds
+\int_{t}^{1}\frac{t(1-s)^{\alpha-1}}{\Gamma(\alpha)}h(s)ds\bigg)dA(t)\\
&=&\int_{0}^{1}\mathcal{G}(t,s)h(s)ds+\frac{\mu t}{1-\Lambda}\int_{0}^{1}\mathcal{G}(\eta,s)h(s)ds
+\frac{\beta t}{1-\Lambda}\int_{0}^{1}\int_{0}^{1}\mathcal{G}(t,s)h(s)ds dA(t)\\
&=&\int_{0}^{1}\mathcal{H}(t,s)h(s)ds.
\end{eqnarray*}
\end{proof}

\begin{lemma}\label{lem 2.6} The Green's function $\mathcal{G}(t,s)$ defined by \eqref{eq12} satisfies
\begin{itemize}
\item[(i)] $ \mathcal{G}(t,s) \geq 0 $, for all\ $t, s \in [0,1],$
\item[(ii)] $(t-t^{\alpha-1})\Psi(s) \leq \mathcal{G}(t,s)\leq \Psi(s),$  \ for all \  $(t,s) \in [0,1] \times [0,1],$
where \[\Psi(s)=\frac{(1-s)^{\alpha-1}}{\Gamma(\alpha)}.\]
\end{itemize}
\end{lemma}
\begin{proof}
\rm{(i)} If $0\leq s \leq t$, then
\begin{eqnarray*}
\Gamma(\alpha)\mathcal{G}(t,s)& =& t(1-s)^{\alpha-1}-(t-s)^{\alpha-1}\\
& =& t(1-s)^{\alpha-1}-t^{\alpha-1}\bigg(1-\frac{s}{t}\bigg)^{\alpha-1}\\
& \geq&t(1-s)^{\alpha-1}-t^{\alpha-1}(1-s)^{\alpha-1}\\
& =& (t-t^{\alpha-1})(1-s)^{\alpha-1}\\
& \geq& 0,\ \ \text{since}\ (t-t^{\alpha-1})\geq0,\ \text{for all}\ t\in[0,1].
\end{eqnarray*}
For $s\geq t$, we have $\Gamma(\alpha)\mathcal{G}(t,s)=t(1-s)^{\alpha-1}\geq 0.$\\
So,
$$\mathcal{G}(t,s)\geq 0,\ \text{for all}\ t, s \in [0,1].$$
\rm{(ii)}
If $s = 1$, then the result follows immediately. Now we suppose that $(t, s)\in[0,1]\times[0, 1)$.
If $t\leq s$, then
\[\Gamma(\alpha)\mathcal{G}(t,s)=t(1-s)^{\alpha-1}\leq(1-s)^{\alpha-1}.\]
So,
\[\mathcal{G}(t,s)\leq\frac{1}{\Gamma(\alpha)}(1-s)^{\alpha-1}=\Psi(s).\]
On the other hand, we have
\[\Gamma(\alpha)\mathcal{G}(t,s)=t(1-s)^{\alpha-1}\geq(t-t^{\alpha-1})(1-s)^{\alpha-1}.\]
So,
\[\mathcal{G}(t,s)\geq \Psi(s)(t-t^{\alpha-1}).\]
If $s\leq t$, then
\[\frac{\mathcal{G}(t,s)}{\Psi(s)}=\frac{t(1-s)^{\alpha-1}-(t-s)^{\alpha-1}}{(1-s)^{\alpha-1}}\leq\frac{t(1-s)^{\alpha-1}}{(1-s)^{\alpha-1}}\leq1.\]
So,
\[\mathcal{G}(t,s)\leq \Psi(s).\]
On the other hand, we have
\begin{eqnarray*}
\frac{\mathcal{G}(t,s)}{\Psi(s)}& =&\frac{t(1-s)^{\alpha-1}-(t-s)^{\alpha-1}}{(1-s)^{\alpha-1}}\\
& \geq&\frac{(t-t^{\alpha-1})(1-s)^{\alpha-1}}{(1-s)^{\alpha-1}}\\
& =& (t-t^{\alpha-1}).
\end{eqnarray*}
So,
\[\mathcal{G}(t,s)\geq \Psi(s)(t-t^{\alpha-1}).\]
Thus,
\[(t-t^{\alpha-1})\Psi(s) \leq \mathcal{G}(t,s)\leq\Psi(s),\ \text{for all}\ t, s \in [0,1].\]
\end{proof}

In the remainder of this paper, we always assume that
\begin{itemize}
 \item[(H1)] $0\leq\Lambda<1;$
\item[(H2)] $A$ is a function of bounded variation, and $\textsl{g}_{A}(s)\geq0,\ \text{for all}\ s\in[0,1]$, where $\textsl{g}_{A}(s)=\int_{0}^{1}\mathcal{G}(t,s)dA(t).$
\end{itemize}

\begin{lemma}\label{lem 2.7} $ \mathcal{H}(t,s)$ defined by \eqref{eq11} satisfies
\begin{itemize}
\item[(i)] $ \mathcal{H}(t,s) \geq 0 $, for all\ $t, s \in [0,1],$
\item[(ii)] $\rho(t)\Phi(s) \leq \mathcal{H}(t,s)\leq \Phi(s),$  \ for all \  $(t,s) \in [0,1] \times [0,1],$
where $$ \Phi(s)=\frac{\beta}{1-\Lambda} \textsl{g}_{A}(s)+\frac{\mu-\Lambda+1}{1-\Lambda}\Psi(s),$$  and
\[\rho(t)=(\eta-\eta^{\alpha-1})(t-t^{\alpha-1}).\]
\end{itemize}
\end{lemma}
\begin{proof}
\begin{eqnarray*}
\mathcal{H}(t,s)& =&\frac{\beta t}{1-\Lambda} \textsl{g}_{A}(s)+\frac{\mu t}{1-\Lambda}\mathcal{G}(\eta,s)+\mathcal{G}(t,s)\\
&\leq&\frac{\beta }{1-\Lambda} \textsl{g}_{A}(s)+\frac{\mu}{1-\Lambda}\Psi(s)+\Psi(s)\\
&\leq& \frac{\beta}{1-\Lambda} \textsl{g}_{A}(s)+\frac{\mu-\Lambda+1}{1-\Lambda}\Psi(s).
\end{eqnarray*}
\begin{eqnarray*}
\mathcal{H}(t,s)& =&\frac{\beta t}{1-\Lambda} \textsl{g}_{A}(s)+\frac{\mu t}{1-\Lambda}\mathcal{G}(\eta,s)+\mathcal{G}(t,s)\\
&\geq&\frac{\beta (t-t^{\alpha-1})}{1-\Lambda} \textsl{g}_{A}(s)+\frac{\mu (t-t^{\alpha-1})}{1-\Lambda}\Psi(s)(\eta-\eta^{\alpha-1})+\Psi(s)(t-t^{\alpha-1})\\
&\geq& \frac{\beta (t-t^{\alpha-1})}{1-\Lambda} \textsl{g}_{A}(s)(\eta-\eta^{\alpha-1})+\frac{\mu (t-t^{\alpha-1})}{1-\Lambda}\Psi(s)(\eta-\eta^{\alpha-1})\\
&&+(\eta-\eta^{\alpha-1})\Psi(s)(t-t^{\alpha-1})\\
&\geq& \bigg(\frac{\beta}{1-\Lambda} \textsl{g}_{A}(s)+\frac{\mu-\Lambda+1}{1-\Lambda}\Psi(s)\bigg)(\eta-\eta^{\alpha-1})(t-t^{\alpha-1}).
\end{eqnarray*}
 \end{proof}

Set $E=C([0,1],\mathbb{R})$. It is well known that $E$ is a Banach space with the norm $\|u\|=\\sup_{t\in[0, 1]}|u(t)|$. Let $\mathcal{P}= \left\{ u \in E : u(t)\geq0,\ t\in[0,1]\right\}$. It is easy to see that $\mathcal{P}$ is a cone in $E$. Let us define a nonlinear operator $\mathcal{A} : \mathcal{P} \rightarrow E$ by
 \begin{equation}\label{eq15}
\mathcal{A}u(t)= \int_{0}^{1} \mathcal{H}(t,s) f(s , u(s)) ds ,\ u\in E,
 \end{equation}
 where $ \mathcal{H}(t,s) $ is defined by \eqref{eq11}. Then the existence of a positive solution of the boundary value problem \eqref{eq8} is equivalent to the existence of a nontrivial fixed point of $\mathcal{A}$ on $\mathcal{P}$.
For $a>0$, we define linear operator $\mathcal{K}_{a}$ by
\[\mathcal{K}_{a}=a\mathcal{K},\ \text{where}\ (\mathcal{K}u)(t)=\int_{0}^{1}\mathcal{H}(t,s)u(s)ds,\ u\in E.\]
Clearly, $\mathcal{K}_{a}:\mathcal{P}\rightarrow \mathcal{P}$ is a completely continuous linear operator. By virtue of the Arzela-Ascoli theorem, $ \mathcal{A}$ is completely continuous and satisfies $ \mathcal{A}\mathcal{P} \subset \mathcal{P}$.

\begin{lemma}\label{lem 2.8}
 The spectral radius of the operator $\mathcal{K}_{a}$ satisfies
\begin{equation}\label{eq16}
a\tau_{1}\leq r(\mathcal{K}_{a})\leq a\tau_{2},
\end{equation}
where
\begin{align}
\tau_{1}&=\frac{\eta-\eta^{\alpha-1}}{1-\Lambda}\int_{0}^{1}(s-s^{\alpha-1})\bigg(\beta \textsl{g}_{A}(s)+(\mu-\Lambda+1)\Psi(s)\bigg)ds \label{eq17}\\
\tau_{2}&=\frac{1}{1-\Lambda}\bigg(\beta\int_{0}^{1}
\textsl{g}_{A}(s)ds+\frac{\mu-\Lambda+1}{\Gamma(\alpha+1)}\bigg)\label{eq18}
\end{align}
\end{lemma}

\begin{proof}
It is easy to show that $\|\mathcal{K}\|=\max_{t\in[0,1]}\int_{0}^{1}\mathcal{H}(t,s)ds.$ By Lemma \ref{lem 2.7} and the definition of $\mathcal{G}(t,s)$, we have
\begin{equation*}
\begin{split}
 \|\mathcal{K}\|&=\max_{t\in[0,1]}\int_{0}^{1}\mathcal{H}(t,s)ds\\
&\leq \int_{0}^{1}\Phi(s)ds\\
&= \frac{\beta}{1-\Lambda}\int_{0}^{1}\textsl{g}_{A}(s)ds+\frac{\mu-\Lambda+1}{1-\Lambda}\int_{0}^{1}\Psi(s)ds\\
&=\tau_{2}.
\end{split}
\end{equation*}
Hence for any $n\in \mathbb{N}^{\star}$, using the recursive approach, we get
\begin{equation*}
\begin{split}
 \|\mathcal{K}^{n}\|&=\max_{t\in[0,1]}\underbrace{\int_{0}^{1}\int_{0}^{1}\ldots \int_{0}^{1}}_{n\ \mathrm{times}}\mathcal{H}(t,s_{n-1})
\mathcal{H}(s_{n-1},s_{n-2})\ldots\mathcal{H}(s_{1},s)ds_{n-1}ds_{n-2}\ldots ds\\
&\leq (\tau_{2})^{n}.
\end{split}
\end{equation*}
By this inequality and the Gelfand formula of spectral radius, we obtain that
\[r(\mathcal{K})=\lim_{n\rightarrow\infty}\sqrt[n]{\|\mathcal{K}^{n}\|}\leq \tau_{2}.\]
On the other hand, for all $n\in \mathbb{N}^{\star}$, we have
\begin{equation*}
\begin{split}
 \|\mathcal{K}^{n}\|&=\max_{t\in[0,1]}\underbrace{\int_{0}^{1}\int_{0}^{1}\ldots \int_{0}^{1}}_{n\ \mathrm{times}}\mathcal{H}(t,s_{n-1})
\mathcal{H}(s_{n-1},s_{n-2})\ldots\mathcal{H}(s_{1},s)ds_{n-1}ds_{n-2}\ldots ds\\
&\geq\max_{t\in[0,1]}\underbrace{\int_{0}^{1}\int_{0}^{1}\ldots \int_{0}^{1}}_{n\ \mathrm{times}}\rho(t)\Phi(s_{n-1})
\rho(s_{n-1})\Phi(s_{n-2})\ldots\rho(s_{1})\Phi(s)ds_{n-1}ds_{n-2}\ldots ds\\
&\geq \max_{t\in [0,1]}\rho(t)\int_{0}^{1}\Phi(s)ds \bigg(\int_{0}^{1}\rho(s)\Phi(s)ds\bigg)^{n-1}\\
&\geq \tau_{2}(\eta-\eta^{\alpha-1})\frac{\alpha-2}{(\alpha-1)^{\frac{\alpha-1}{\alpha-2}}}\bigg(\int_{0}^{1}\rho(s)\Phi(s)ds\bigg)^{n-1}.
\end{split}
\end{equation*}
From Gelfand's formula, we get
\[r(\mathcal{K})=\lim_{n\rightarrow\infty}\sqrt[n]{\|\mathcal{K}^{n}\|}\geq \bigg(\int_{0}^{1}\rho(s)\Phi(s)ds\bigg)=\tau_{1}.\]
\end{proof}

\section{Main result}
By imposing appropriate assumptions upon nonlinearity $f(t,u)$, we can prove the following existence result for problem \eqref{eq8}.
\begin{theorem}\label{th1}
Let $f:[0,1]\times [0,\infty)\rightarrow [0,\infty)$ be continuous. If $f$ satisfies the following
conditions
\begin{itemize}
\item[(C1)] There exist $a\in(0,\tau_{2}^{-1})$ and $c>0$ such that
\[f(t,x)\leq ax+c,\ \text{for}\ (t,x)\in [0,1]\times[0,\infty);\]
\item[(C2)] There exist $b\in[\tau_{1}^{-1},\infty)$ and $\delta>0$ such that
\[f(t,x)\geq bx,\ \text{for}\ (t,x)\in [0,1]\times[0,\delta],\]
\end{itemize}
where $\tau_{1}$ and $\tau_{2}$ are given by \eqref{eq17}-\eqref{eq18}. Then BVP \eqref{eq8} has at least one positive solution.
\end{theorem}
\begin{proof}
Let $W=\{u\in \mathcal{P}: u=\lambda \mathcal{A}u, \lambda\in[0,1]\}$. We will prove that $W$ is a bounded set in $P$. If $u\in W$,
then from \rm{(C1)} we have
\begin{equation*}
\begin{split}
u(t)&=\lambda \mathcal{A}u(t)\\
&\leq  \int_{0}^{1} \mathcal{H}(t,s)f(s,u(s))ds \\
&\leq  \int_{0}^{1} \mathcal{H}(t,s)(au(s)+c)ds \\
&= a \int_{0}^{1} \mathcal{H}(t,s)u(s)ds+c\int_{0}^{1} \mathcal{H}(t,s)ds \\
& = (\mathcal{K}_{a}u)(t)+c \sigma(t),
\end{split}
\end{equation*}
where
\begin{equation*}
\begin{split}
\sigma(t)&=\int_{0}^{1} \mathcal{H}(t,s)ds\\
&=\frac{ t}{1-\Lambda}\int_{0}^{1}\big(\beta \textsl{g}_{A}(s)+\mu \mathcal{G}(\eta,s)\big)ds+\int_{0}^{1}\mathcal{\mathcal{G}}(t,s)ds,\ t\in[0,1].
\end{split}
\end{equation*}
So,
\[((I-\mathcal{K}_{a})u)(t)\leq c \sigma(t).\]
The condition $a\in(0,\tau_{2}^{-1})$, implies that $r(\mathcal{K}_{a})<1$, and so $(I-\mathcal{K}_{a})$ has a bounded inverse operator
$(I-\mathcal{K}_{a})^{-1}$ which is given by
\[(I-\mathcal{K}_{a})^{-1}=I+\mathcal{K}_{a}+\mathcal{K}_{a}^{2}+...\mathcal{K}_{a}^{n}+....\]
Since $\mathcal{K}_{a}(\mathcal{P})\subset\mathcal{P}$, it follows that
$(I-\mathcal{K}_{a})^{-1}(\mathcal{P})\subset\mathcal{P}$. So we have
\begin{equation}\label{eq19}
u(t)\leq ((I-\mathcal{K}_{a})^{-1}c\sigma)(t),\ \text{for}\ t\in [0,1].
\end{equation}
By \eqref{eq19}, we have $\|u\|\leq \|((I-\mathcal{K}_{a})^{-1}c\sigma)\|$, that is, $W$ is bounded.\\
By selecting $R>\max\big\{\delta,\ \sup\{\|u\|:u\in W\}\big\}$ and $\Omega_{R}=\{u\in E:\|u\|<R\}$, then we have
\[ u \neq \lambda \mathcal{A}u,\ \text{for}\ u\in\partial\Omega_{R}\cap\mathcal{P},\ \lambda\in[0,1].\]
From Lemma \ref{lem 2.2}, we have

\begin{equation}\label{eq20}
i\big(\mathcal{A},\Omega_{R}\cap\mathcal{P},\mathcal{P}\big)=1.
\end{equation}
On the other hand, by \rm{(C2)} we have $r(\mathcal{K}_{b})\geq1$, and since
 $\mathcal{K}_{b}(\mathcal{P})\subset\mathcal{P}$, then by Lemma \ref{lem 2.4}, we know that there exists
 $\varphi_{0}\in \mathcal{P}\setminus\{0\}$ such that $\mathcal{K}_{b}\varphi_{0}=r(\mathcal{K}_{b})\varphi_{0}$ and
$\varphi_{0}=r(\mathcal{K}_{b})^{-1}\mathcal{K}_{b}\varphi_{0} \in \mathcal{P}$.
Now we show that
\begin{equation}\label{eq21}
u-\mathcal{A}u\neq \lambda \varphi_{0},\ \text{for all}\ u\in\partial\Omega_{\delta}\cap\mathcal{P},\ \lambda\geq0,
\end{equation}
where \[\Omega_{\delta}=\{u\in E:\|u\|<\delta\}.\]
We may suppose that $\mathcal{A}$ has no fixed points on
 $\partial\Omega_{\delta}\cap\mathcal{P}$ (otherwise, the proof is finished). Assume by contradiction that there exist $u_{0}\in \partial\Omega_{\delta}\cap\mathcal{P}$ and $\lambda_{0}\geq0$ such that $u_{0}-\mathcal{A}u_{0}=\lambda_{0}\varphi_{0}$, then $\lambda_{0}>0$, and by $\left(C_{2}\right)$, it follows that $(\mathcal{A}u_{0})(t)\geq (\mathcal{K}_{b}u_{0})(t)$. Thus
\[u_{0}=\mathcal{A}u_{0}+\lambda_{0}\varphi_{0}\geq \mathcal{K}_{b}u_{0}+\lambda_{0}\varphi_{0}\geq\lambda_{0}\varphi_{0}.\]
Let $\lambda^{\ast}=\sup\{\lambda: u_{0}\geq \lambda\varphi_{0}\}$, then $\lambda^{\ast}>0$ and $u_{0}\geq \lambda^{\star}\varphi_{0}$.
Since $\mathcal{K}_{b}(\mathcal{P})\subset\mathcal{P}$, we have
\begin{equation*}
\begin{split}
u_{0}&\geq \mathcal{K}_{b}u_{0}+\lambda_{0}\varphi_{0}\\
&\geq \mathcal{K}_{b}\lambda^{\star}\varphi_{0}+\lambda_{0}\varphi_{0}\\
&=\lambda^{\star}r(\mathcal{K}_{b})\varphi_{0}+\lambda_{0}\varphi_{0}\\
&=(\lambda^{\star}r(\mathcal{K}_{b})+\lambda_{0})\varphi_{0},
\end{split}
\end{equation*}
which contradicts the definition of $\lambda^{\star}$. So \eqref{eq21} is true and by Lemma \ref{lem 2.3} we have
\begin{equation}\label{eq22}
i\big(\mathcal{A},\Omega_{\delta}\cap\mathcal{P},\mathcal{P}\big)=0.
\end{equation}
By \eqref{eq22}and \eqref{eq20}, we get
\[i\big(\mathcal{A},(\Omega_{R}\setminus\overline{\Omega}_{\delta})\cap\mathcal{P},\mathcal{P}\big)=
i\big(\mathcal{A},\Omega_{R}\cap\mathcal{P},\mathcal{P}\big)-i\big(\mathcal{A},\Omega_{\delta}\cap\mathcal{P},\mathcal{P}\big)=1.\]
Therefore, $\mathcal{ A}$ has at least one fixed point in $(\Omega_{R}\setminus\overline{\Omega}_{\delta})\cap\mathcal{P}$ which means that
boundary value problem \eqref{eq8} has at least one positive solution.

\end{proof}

\begin{exmp}
Consider the boundary value problem \eqref{eq8} with $\alpha=\frac{5}{2}$, $\beta=1$, $ \mu=2$, $\eta=\frac{1}{7}$, $f(t,u) =1-t+e^{\frac{t}{4}-u}$, and

\begin{equation*}
A(t)=\begin{cases} 0, & t\in [0, \frac{3}{7}); \\
2 , & t\in [\frac{3}{7}, \frac{4}{7}),\\
1 , & t\in [\frac{4}{7}, 1]. \end{cases}
\end{equation*}
Putting
\begin{equation*}                                                                                               \mathcal{G}(t,s)=\frac{1}{\Gamma(\alpha)} \begin{cases}\mathcal{G}_{1}(t,s), & 0 \leq s \leq t \leq 1, \\
\mathcal{G}_{2}(t,s) , & 0 \leq t \leq s \leq 1.
\end{cases}
\end{equation*}
Thus
\begin{equation*}
\textsl{g}_{A}(s)=\begin{cases} 2\mathcal{G}_{1}(\frac{3}{7},s)-\mathcal{G}_{1}(\frac{4}{7},s), & 0\leq s<\frac{3}{7},\\
 2\mathcal{G}_{2}(\frac{3}{7},s)-\mathcal{G}_{1}(\frac{4}{7},s), & \frac{3}{7}\leq s<\frac{4}{7},\\
2\mathcal{G}_{2}(\frac{3}{7},s)-\mathcal{G}_{2}(\frac{4}{7},s), & \frac{4}{7}\leq s\leq 1.\\
\end{cases}
\end{equation*}
Consequently
\begin{equation*}                                                                                                     \textsl{g}_{A}(s)=\frac{1}{\Gamma\big(\frac{5}{2}\big)}\begin{cases}
\frac{2}{7}(1-s)^{\frac{3}{2}}-2(\frac{3}{7}-s)^{\frac{3}{2}}+(\frac{4}{7}-s)^{\frac{3}{2}}, & 0\leq s<\frac{3}{7},\\
\frac{2}{7}(1-s)^{\frac{3}{2}}+(\frac{4}{7}-s)^{\frac{3}{2}}, & \frac{3}{7}\leq s<\frac{4}{7},\\
\frac{2}{7}(1-s)^{\frac{3}{2}}, & \frac{4}{7}\leq s\leq 1.
\end{cases}
\end{equation*}
So, $\Lambda=\mu\eta+\beta\int_{0}^{1}tdA(t)=2\frac{1}{7}+2\frac{3}{7}-\frac{4}{7}=\frac{4}{7}<1$, and $\textsl{g}_{A}(s)\geq0$. The boundary value problem \eqref{eq8} becomes the fractional five-point boundary value problem

\begin{equation}\label{eq23}
  \begin{cases}  D^{\frac{5}{2}}u(t) +f(t,u(t)) = 0,\  t \in (0,1),\\
 u(0)= u^{\prime\prime} (0) =0,\ u(1)=2 u(\frac{1}{7})+ 2u(\frac{3}{7})-u(\frac{4}{7}).
  \end{cases}
  \end{equation}

On the other hand, by calculation, we have
\begin{align*}
\tau_{1}^{-1}&=\Bigg(\frac{\eta-\eta^{\alpha-1}}{1-\Lambda}\int_{0}^{1}(s-s^{\alpha-1})\bigg(\beta \textsl{g}_{A}(s)+(\mu-\Lambda+1)\Psi(s)\bigg)ds\Bigg)^{-1} \approx 57.3423,\\
\tau_{2}^{-1}&=\Bigg(\frac{1}{1-\Lambda}\bigg(\beta\int_{0}^{1}
\textsl{g}_{A}(s)ds+\frac{\mu-\Lambda+1}{\Gamma(\alpha+1)}\bigg)\Bigg)^{-1} \approx 0.523515.
\end{align*}

Next, let us choose $a=\frac{2}{5}$, $c=3$, $b=58$, and $\delta=\frac{3}{200}$. Using the Mathematica software, we easily check that
\[f(t,u)\leq 1+e^{\frac{1}{4}-u}\leq\frac{2}{5}u+3,\ \text{for all}\ (t,u)\in[0,1]\times[0,\infty),\]
\[f(t,u)\geq e^{\frac{t}{4}-u}\geq e^{-u}\geq 58u,\ \text{for all}\ (t,u)\in[0,1]\times\Big[0,\frac{3}{200}\Big],\]

So $f(t,u)$ satisfies the conditions \rm{(H1)} and \rm{(H2)}. Consequently, by Theorem \ref{th1}, the problem \eqref{eq23} has at least one positive solution.
\end{exmp}


\begin{thebibliography}{99}

\bibitem{Hao} X. Hao, L. Liu, Y. Wu, Q.
Sun, {\it Positive solutions for nonlinear $n$th-order singular eigenvalue problem with nonlocal conditions}, Nonlinear Anal., \textbf{73}(6) (2010), 1653--1662.

\bibitem{Hao2} X. Hao, H. Wang, {\it Positive solutions of semipositone singular fractional differential systems with a parameter and integral boundary conditions}, Open Math. \textbf{16} (1) (2018), 581--596.



\bibitem{Wang2} Y. Wang, L. Liu, Y. Wu, {\it Positive solutions  for a nonlocal fractional differential equation},
Nonlinear Anal., \textbf{74}(11) (2011),
3599--3605.

\bibitem{Song} Q. Song, Z. Bai, {\it Positive solutions of fractional differential
equations involving the Riemann–Stieltjes integral boundary condition}, Adv. Difference Equ., \textbf{183} (2018), 7 pp.
\bibitem{Lv} T. Lv, H. Pang, L.Cao, {\it Existence Results for Fractional Differential Equations with
 Multistrip Riemann-Stieltjes Integral Boundary Conditions}, Discrete Dyn. Nat. Soc., \textbf{2018} (2018), Article ID 2352789, 8 pages.

\bibitem{Zhang1} X. Zhang, X. Liu, M. Jia, H. Chen, {\it The Positive Solutions of Fractional Differential Equation with Riemann-Stieltjes Integral Boundary Conditions}, Filomat. \textbf{32}(7) (2018), 2383–-2394.

\bibitem{Zhang2} X. Zhang, L. Liu, Y. Wu, B. Wiwatanapataphee, {\it The spectral analysis for a singular fractional differential equation with a signed measure}, Appl. Math. Comput., \textbf{257} (2015), 252--263.



\bibitem{Zhang3} X. Zhang, L. Liu, Y. Wu,  {\it The uniqueness of positive solution for a fractional order model of turbulent flow in a porous medium}, Appl. Math. Lett.,\textbf{37} (2014), 26--33.


\bibitem{Cabada} A. Cabada, G. Wang, {\it Positive solutions of nonlinear fractional differential
equations with integral boundary value conditions}, J. Math. Anal. Appl. \textbf{389} (1) (2012), 403–411.


\bibitem{Guo} D. Guo, V. Lakshmikantham, {\it Nonlinear Problems in Abstract Cones}, Academic Press Inc, New York, 1988.

\bibitem{Deimling} K. Deimling, Nonlinear Functional Analysis, Springer-Verlag, Berlin, 1985.

\bibitem{Kilbas2006}
A. A. Kilbas, H. M. Srivastava, J. J. Trujillo, {\it Theory and applications of fractional differential equations}, Elsevier, Amsterdam, The Netherlands, 2006.

\bibitem{Debowska} K. Szyma\'{n}ska-D\c{e}bowska, M. Zima, {\it A topological degree approach to a nonlocal Neumann problem for a system at resonance}, J. Fixed Point Theory Appl., \textbf{21} (2) (2019), Art. 67, 14 pp.

\bibitem{Tan} J. Tan, {\it Positive solutions of singular fractional order differential system with Riemann-Stieltjes integral boundary condition}, Adv. Difference Equ., (293) (2016), 12 pp.


\bibitem{Ren} T. Ren, S. Li, X. Zhang, L. Liu, {\it Maximum and minimum solutions for a nonlocal $p$-Laplacian fractional differential system from eco-economical processes}, Bound. Value Probl.  (118) (2017), 15 pp.

\bibitem{Jia} M. Jia, L. Li, X. Liu, J. Song, Z. Bai, {\it A cla
ss of nonlocal problems of fractional
differential equations with composition of derivative and parameters}, Adv. Difference Equ., (280) (2019), 26 pp.

\bibitem{Srivastava} H. M. Srivastava, A. M. A. El-Sayed, F. M. Gaafar, {\it A Class of Nonlinear Boundary Value Problems for an Arbitrary Fractional-Order Differential Equation with the Riemann-Stieltjes Functional Integral and Infinite-Point Boundary Conditions}, Symmetry-Basel, \textbf{10} (10) (2018), 1--13.

   \bibitem{Ahmed} B. Ahmad, A. Alsaedi, S. Salem, S. K. Ntouyas, {\it Fractional Differential Equation Involving Mixed Nonlinearities with Nonlocal Multi-Point and Riemann-Stieltjes Integral-Multi-Strip Conditions}, Fractal Fract., \textbf{3} (2) (2019), 34, 1-16.

   \bibitem{Padhi} S. Padhi, J. R. Graef, S. Pati, {\it Multiple positive solutions for a boundary value problem with nonlinear nonlocal Riemann-Stieltjes Integral boundary conditions}, Fract. Calc. Appl. Anal., \textbf{21} (3) (2018), 716--745.


\bibitem{Jankowski1} T. Jankowski, {\it Positive solutions to fractional differential equations involving Stieltjes integral conditions}, Appl. Math. Comput., \textbf{241} (2014), 200–-213.

\bibitem{Jankowski2} T. Jankowski, {\it Positive solutions to second-order differential equations with the dependence on the first order derivative and nonlocal boundary conditions},  Bound. Value Probl., \textbf{8} (2013), 21 pp.

\bibitem{Wang}  Y. Wang, {\it Positive solutions for fractional differential equation involving Riemann-Stieltjes integral conditions with two parameters}, J. Nonlinear Sci. Appl. \textbf{9} (11) (2016), 5733–-5740.

\bibitem{Kaihong} Z. Kaihong, G. Ping, {\it Positive solutions of Riemann-Stieltjes integral boundary problems for the nonlinear coupling system involving fractional-order differential}, Adv. Difference Equ., \textbf{254} (2014), 18 pp.

  \bibitem{Kossowski} I. Kossowski, K. Szyma\'{n}ska-D\c{e}bowska, {\it  Solutions to resonant boundary value problem with boundary conditions involving Riemann-Stieltjes integrals}, Discrete Contin. Dyn. Syst. Ser. B, \textbf{23} (1) (2018), 275--281.

\bibitem{Jiang} J. Jiang, L. Liu, Y. Wu, {\it Positive solutions for nonlinear fractional differential equations with boundary conditions involving Riemann-Stieltjes integrals}, Abstr. Appl. Anal., (2012), Art. ID 708192, 21 pp.




\bibitem{Tarasov} V. Tarasov, {\it Lattice model of fractional gradient and integral elasticity: long-range interaction of Gr\"{u}nwald-Letnikov-Riesz type}, Mech. Mater., \textbf{70} (2014), 106--114.

 \bibitem{Yin} D. Yin, X. Duan, X. Zhou, {\it Fractional time-dependent deformation component models for characterizing viscoelastic Poisson's ratio}, Eur. J. Mech. A, Solids., \textbf{42} (2013), 422--429.


  \bibitem{Paola} M. Paola, F. Pinnola, M. Zingales, {\it Fractional differential equations and related exact mechanical models}, Comput. Math. Appl., \textbf{66} (2013), 608--620.


 \bibitem{Hartley} T. Hartley, C. Lorenzo, H. Qammer, {\it Chaos in fractional order Chua's system}, IEEE Trans. Circuits Syst. I, Fundam. Theory Appl. \textbf{42} (8) (1995), 485--490.




 \bibitem{Caputo} M. Caputo, {\it Free modes splitting and alterations of electro chemically polarizable media}, Rend Fis. Accad. Lincei., \textbf{4} (1993), 89--98.


\bibitem{Saka} H. El-Saka,  {\it The fractional-order SIS epidemic model with variable population size}, J. Egypt. Math. Soc. \textbf{22} (2014), 50--54.


 \bibitem{Caffarelli} L. Caffarelli, J. Vazquez, {\it Nonlinear porous medium flow with fractional potential pressure}, Arch. Ration. Mech. Anal., \textbf{202} (2011), 537--565.

 \bibitem{Ivancevic} V. G. Ivancevic, T. T. Ivancevic, {\it Geometrical Dynamics of Complex Systems: A Unified Modelling Approach to Physics, Control, Biomechanics, Neurodynamics and Psycho-Socio-Economical Dynamics}, Springer, Berlin (2006).


 \bibitem{Das1} S. Das, K. Maharatna, {\it Fractional dynamical model for the generation of ECG like signals from filtered coupled Van-der Pol oscillators}, Comput. Methods Programs Biomed., \textbf{122} (2013), 490--507.

\bibitem{Das2} S. Das, {\it Functional fractional calculus for systems identification and controls}, Springer, New york, 2008.


 \bibitem{Mongiovi} M. Mongiovi, M. Zingales, {\it A non-local model of thermal energy transport: the fractional temperature equation}, Int. J. Heat Mass Transf., \textbf{67} (2013), 593--601.



 \bibitem{Zhang4} X. Zhang, L. Liu, Y. Wu, B. Wiwatanapataphee,  {\it Nontrivial solutions for a fractional advection dispersion equation in anomalous diffusion}, Appl. Math. Lett., \textbf{66} (2017), 1--8.


 \bibitem{Zhang5} X. Zhang, C. Mao, L. Liu, Y. Wu, {\it Exact iterative solution for an abstract fractional dynamic system model for bioprocess}, Qual. Theory Dyn. Syst., \textbf{16} (2017), 205--222.

 \bibitem{Sabatier} J. Sabatier, O. P. Agarwal, J. A. T. Machado (Eds.) {\it Advances in Fractional calculus: Theoretical Developments and Applications in Physics and Engineering}, Springer, Dordrecht, 2007.

  \bibitem{Beleanu} D. Beleanu, K. Diethelm, E. Scalas, J. J. Trujillo, {\it Fractional calculus models and numerical methods}, Series on complexity, Nonlinearity and chaos, Wold scientific, Boston, 2012.

   \bibitem{Podlubny} I. Podlubny, {\it Fractional differential equations}, Mathematics in sciences and engineering, Academic Press, New york, London, Toronto, 1999.



\end{thebibliography}
\end{document}